\documentclass[a4paper, 11pt]{article}
\usepackage{anysize}
\marginsize{3,5cm}{2,5cm}{2,5cm}{2,5cm}
\usepackage{amssymb}
\usepackage{amsmath,amsbsy,amssymb,amscd}
\usepackage{t1enc}
\usepackage[cp1250]{inputenc}
\usepackage[british]{babel}
\usepackage[all]{xy}
\usepackage{color}
\usepackage{amsfonts}
\usepackage{latexsym}
\usepackage{amsthm}
\usepackage{mathrsfs}
\usepackage{hyperref}
\usepackage{graphicx}

\usepackage{color}
\usepackage{caption}
\usepackage{subcaption}
\usepackage{enumerate}
\usepackage{indentfirst}

\definecolor{orange}{rgb}{1,0.5,0}

\usepackage{mathrsfs} 

\DeclareMathAlphabet{\mathpzc}{OT1}{pzc}{L}{it} 

\newtheorem{definition}{Definition}[section]

\newtheorem{proposition}[definition]{Proposition}
\newtheorem{theorem}{Theorem}

\newtheorem{corollary}[definition]{Corollary}
\newtheorem{remark}[definition]{Remark}
\newtheorem{lemma}[definition]{Lemma}

\def\R{\mathbb{R}}
\def\T{\mathbb{T}}

\def\Z{\mathbb{Z}}
\def\N{\mathbb{N}}

\def\Q{\mathbb{Q}}
\def\cB{\mathcal{B}}

\def\cT{\mathcal T}
\def\cR{\mathcal R}

\def\cC{\mathcal{C}}

\newcommand{\bea}{\begin{eqnarray}}
  \newcommand{\eea}{\end{eqnarray}}
  \newcommand{\beab}{\begin{eqnarray*}}
  \newcommand{\eeab}{\end{eqnarray*}}

  \newcommand{\be}{\begin{equation}}
  \newcommand{\ee}{\end{equation}}

\newcommand{\cD}{\mathcal D}

 \newcommand{\zdk}{(Z,\mathcal{D},\kappa)}

\title{Rigidity of a class of smooth singular flows on $\T^2$}
\author{Changguang Dong and Adam Kanigowski}
\date{\today}
\begin{document}
\baselineskip=14pt \maketitle
\begin{abstract}
We study {\it joining rigidity} in the class of von Neumann flows with one singularity. They are given by a smooth vector field $\mathcal{X}$ on $\T^2\setminus \{a\}$, where $\mathcal{X}$ is not defined at $a\in \T^2$. It follows that the phase space can be decomposed into a (topological disc) $D_\mathcal{X}$ and an ergodic component $E_\mathcal{X}=\T^2\setminus D_\mathcal{X}$. Let $\omega_\mathcal{X}$ be the 1-form associated to $\mathcal{X}$. We show that if $|\int_{E_{\mathcal{X}_1}}d\omega_{\mathcal{X}_1}|\neq
|\int_{E_{\mathcal{X}_2}}d\omega_{\mathcal{X}_2}|$, then the corresponding flows $(v_t^{\mathcal{X}_1})$ and $(v_t^{\mathcal{X}_2})$ are disjoint. It also follows that for every $\mathcal{X}$ there is a uniquely associated frequency $\alpha=\alpha_{\mathcal{X}}\in \T$. We show that for a full measure set of $\alpha\in \T$ the class of smooth time changes of $(v_t^\mathcal{X_\alpha})$  is {\em joining rigid}, i.e.\ every two smooth time changes are either cohomologous or disjoint. This gives a natural class of flows for which the answer to Problem 3 in \cite{Rat4} is positive.

\end{abstract}

\tableofcontents

\section{Introduction}
This paper deals with {\em joinings} in the class of von Neumann flows. Von Neumann flows were introduced in \cite{vNe} as the first systems with continuous spectrum (weakly mixing systems). They are given by a smooth vector field $\mathcal{X}$ on $\T^2\setminus\{a_1,\ldots,a_k\}$, where the vector field is not defined (singular) at $a_i$, $i=1,\ldots,k$. We will be interested in the situation where $\mathcal{X}$ has just one singularity $a\in \T^2$. More precisely (see e.g. \cite{FrLem3}), let $p:\T^2\to \R$ be a $C^\infty$, positive function on $\T^2\setminus\{a\}$ and $p(a)=0$. The vector field $\mathcal{X}$ is given by 
$\mathcal{X}:=\frac{\mathcal{X}_H}{p(\cdot)}$, where $\mathcal{X}_H$ is a Hamiltonian vector field (generating a smooth flow $(h_t)$), $\mathcal{X}_H=\left( \frac{\partial H}{\partial y },-\frac{\partial H}{\partial x}\right)$, where $H:\R^2\to \R$ is $1$-periodic and $a\in\T^2$ is (the only) critical point for $H$ (on $\T^2$). Then the von Neumann flow $(v^\mathcal{X}_t)$ is given by the solution of
$$
 \frac{d \bar{x}}{dt}=\mathcal{X}(\bar{x}).
$$
Notice that the orbits of $(v^\mathcal{X}_t)$ and $(h_t)$ are the same (modulo the fixed point of $(h_t)$). Therefore, by \cite{Arn}, it follows that the phase space decomposes into one region $D_\mathcal{X}$ (homeomorphic to the disc) filled with periodic orbits and an ergodic component $E_\mathcal{X}=\T^2\setminus D_\mathcal{X}$. Let $\omega_\mathcal{X}(Y)=\frac{\langle \mathcal{X},Y\rangle}{\langle \mathcal{X},\mathcal{X}\rangle}$ (notice that $\omega_\mathcal{X}$ is $C^\infty(\T^2\setminus \{a\}$)). Our main theorem is the following:

\begin{theorem}\label{th0}Let $(v^{\mathcal{X}_1}_t)$ and $(v^{\mathcal{X}_2}_t)$ be such that $|\int_{E_{\mathcal{X}_1}}d\omega_{\mathcal{ X}_1}|\neq \int_{E_{\mathcal{X}_2}}d\omega_{\mathcal{ X}_2}|$. Then  $(v^{\mathcal{X}_1}_t)$ and $(v^{\mathcal{X}_2}_t)$ are disjoint.
\end{theorem}

An important consequence of Theorem \ref{th0} is related to {\it joining rigidity} of {\it time changes}. Recall that if $(\phi_t)$ is a flow on $M$ generated by a vector field $Z_\phi$ and $\tau:M\to R_{>0}$, then the time changed flow $(\phi_t^\tau)$ is generated by the vector field $\tau(\cdot)Z_\phi$. In \cite{Rat4}, M.\ Ratner established strong rigidity phenomena for $C^1$ time changes of horocycle flows. Namely, Ratner showed that if $\tau_1$ and $\tau_2$ are time changes  of $(h_t^1)$ and $(h_t^2)$ acting respectively on $SL(2,\R)/\Gamma$ and $SL(2,\R)/\Gamma'$, then either the time changed flows $(h^{\tau_1}_{t})$ and $(h^{\tau_2}_{t})$ are disjoint or $\tau_1$ and $\tau_2$ are {\em jointly cohomologous} (see Definition 2 in \cite{Rat4}). Moreover, M.\ Ratner posed a problem (see Problem 3 in \cite{Rat4}) asking whether there are other classes of measure preserving flows for which the class of smooth functions is {\em joining rigid}, i.e. any joining between any smooth time changes is of algebraic nature (Definition 2 in \cite{Rat4}). Recall that the only natural class beyond horocycle flows for which the class of smooth functions is joining rigid is the class of linear flows on $\T^2$ with diophantine frequencies. Indeed, it follows by \cite{Kol} that every two smooth time changes are cohomologous. In this case however, the first part of the alternative (disjointness) can never be observed.  
Theorem \ref{th0} gives an answer to Ratner's problem in a strong sense, moreover one can observe non-trivial joining rigidity phenomena (both cases, i.e. disjointness and cohomology are realizable). Namely we have the following corollary:
\begin{corollary}\label{cor1} There exists a full measure set $\cD\subset \T$ such that for every $\alpha\in \cD$
the flow\footnote{On the ergodic component $E_{\mathcal{X}_\alpha}$.} $(v_t^{\mathcal{X}_\alpha})$ is (strongly) joining rigid; i.e. for any $\psi,\phi\in C^{\infty}(\T^2)$ with 
$\int_{\T^2}\psi=\int_{\T^2}\phi$, either $\psi$ and $\phi$ are cohomologous\footnote{This is equivalent to $\int_{\partial D_\mathcal{X}} \psi(v^{\mathcal{X}_\alpha}_t)dt=\int_{\partial D_\mathcal{X}}\phi(v^{\mathcal{X}_\alpha}_t)dt$.}, or the time changed flows $(v^{\mathcal{X}_\alpha,\psi}_t)$ and $(v^{\mathcal{X}_\alpha,\phi})$ are disjoint.
\end{corollary}
We will give a proof of Corollary \ref{cor1} in Section \ref{proof:cor}.

It turns out that von Neumann flows (with one singularity) can be represented (on the ergodic component) as special flows over irrational rotations and roof functions of bounded variation which are absolutely continuous except one point at which there is a jump discontinuity, which comes from the singularity of the vector field $\mathcal{X}$. More precisely, let $R_\alpha x=x+\alpha\,{\rm mod} \,1$ and let $f:\T\to\R_+$ be given by
$$
f(x)=A_f\{x\}+f_{ac}(x),
$$
where $f_{ac}\in C^1(\T)$.
Then the von Neumann flow $(v_t^{\mathcal{X}_\alpha})$ is isomorphic to the special flow $(T_t^{\alpha,f})$, where $A_f:=\int_{E_{\mathcal{X}_\alpha}}d\omega_\mathcal{X_\alpha}$.
Using the language of special flows, Theorem \ref{th0} is a straightforward consequence of the following theorem:
\begin{theorem}\label{th1}
If $|A_f|\neq |A_g|$, then the flows $\cT=(R^{\alpha,f}_t)$ and $\cR=(R^{\beta,g}_t)$ are disjoint.
\end{theorem}
Notice that there are no assumptions on the irrationals $\alpha,\beta\in \T$ in Theorem \ref{th1}.
The statistical orbit growth of von Neumann flows is linear and hence they exhibit features both from elliptic and parabolic paradigm. On the one hand they are never mixing, \cite{Kat}, and have singular maximal spectral type. On the other hand as shown in \cite{FrLem} and \cite{FrLem2}, if $\alpha$ is of bounded type, they are mildly mixing. Moreover, from \cite{KaSo} it follows that they are never of finite rank, in particular they don't have fast approximation property, \cite{KatSto}. Our methods rely on the parabolic features of von Neumann flows. One of the main ingredients in the proof is a variant of {\em parabolic disjointness} criterion introduced in \cite{KLU}.

\textbf{Plan of the paper.} In Section \ref{sec:def} we recall basic definitions. In Section \ref{sec:disjoint} we recall a variant of disjointness criterion from \cite{KLU}. Section \ref{proof:th1} is devoted for the proof of Theorem \ref{th1}, which we divide in two subsections  (Subsections \ref{sec:un} and \ref{sec:bo}) depending on the diophantine type of $\alpha$ and $\beta$. Finally, in Section \ref{proof:cor} we give a proof of Corollary \ref{cor1}.

\section{Definitions and notations}\label{sec:def}
\subsection{Time changes of flows}
Let $(T_t)$ be a flow on $\zdk$ and let $v\in L^1\zdk$ be a positive function. Then the time change of $(T_t)$ along $v$ is given by
$$
T^v_t(x)=T_{u(t,x)}(x),
$$
where $u:Z\times \R\to \R$ is the unique solution to
$$
\int_{0}^{u(t,x)}\tau(T_sx) ds=t.
$$
Note that the function $u=u(t,x)$ satisfies the cocycle identity: $u(t_1+t_2,x)=u(t_1,x)+u(t_2,T_t^vx)$. The new flow $(T^v_t)$ has the same orbits as the original flow. 
We say that $\psi,\phi \in L^1\zdk$ are {\it cohomologous} if there exists $\xi\in L^1\zdk$ such that for every $t\in \R$
$$
\int_{0}^t \psi(T_sx)-\phi(T_sx)ds=\xi(x)-\xi(T_tx).
$$
It follows that if $\psi,\phi$ are cohomologous, then the flows $(T_t^\psi)$ and $(T_t^\phi)$ are isomorphic.

\subsection{Disjointness, special flows}
Let $(T_t):(X,\cB,\lambda)\to (X,\cB,\lambda)$ and $(S_t):(Y,\cC,\nu)\to(Y,\cC,\nu)$ be two ergodic flows. A {\it joining} between $(T_t)$ and $(S_t)$ is any $(T_t\times S_t)$ invariant probability measure on $X\times Y$  such that
$$
\rho(A\times Y)=\lambda(A) \text{ and } \rho(X\times B)=\nu(B).
$$
We denote the set of joinings by $J((T_t),(S_t))$. We say that $(T_t)$ and $(S_t)$ are disjoint (denoted by $(T_t)\perp(S_t)$), if $J((T_t),(S_t))=\{\lambda\otimes \nu\}$.

We will be interested in disjointness in the class of special flows over irrational rotations. For $\alpha\in \R\setminus\Q$, let $[0.a_1,a_2,...]$ denote the continued fraction expansion of $\alpha$ and let $(q_n)$ denote the sequence of denominators, i.e. 
$$
q_{n+1}=a_nq_n+q_{n-1},\text{ with }q_0=q_1=1.
$$
We say that $\alpha$ is of {\it bounded type}, if  $\sup_{n\in\N}a_n<M$ for some constant $M>0$ (equivalently, if $q_{n+1}\leq C_\alpha q_n$ for every $n\in\N$), otherwise we say of {\it unbounded type}. The following set will be important in the proof of Corollary \ref{cor1}. Let 
\be\label{dioph}
DC:=\bigcup_{\tau>0}DC(\tau)\subset \T,
\ee
where 
$$
DC(\tau):=\left\{\alpha\in \T: \left|\alpha-\frac{p}{q}\right|>\frac{C_\alpha}{q^\tau}, \text{ for every } p,q\in \N \right\}.
$$

Let $R_\alpha(x)=x+\alpha\,{\rm mod}\, 1$ and $f\in L^1(\T,\cB,\lambda)$. We define the $\Z$-cocycle given by
$$
f^{(n)}(x)=\sum_{i=0}^{n-1}f(R^i_\alpha x) \text{ and } f^{(-n)}(x)=-f^{(n)}(R_\alpha^{-n}x),
$$
 for $n\geq 1$ and we set $f^{(0)}(x)=0$. Let $\T^f:=\{(x,s)\;:\; 0\leq s<f(x)\}$ and let $\cB^f$ and $\lambda^f$ denote respectively the $\sigma$-algebra $\cB\otimes\R$ and the measure $\lambda\otimes Leb$ restricted to $\T^f$.
 We define the special flow 
 $(T_t^f):(\T^f,\cB^f,\lambda^f)\to(\T^f,\cB^f,\lambda^f)$, by 
 $$
 T_t^f(x,s):=(x+N(x,s,t)\alpha, s+t-f^{(N(x,s,t))}(x)).
 $$
 We will consider the product metric on the space $\T^f$, i.e. 
 $$
 d^f((x,s),(y,r):=\|x-y\|+|s-r|.
 $$
For a set $A\subset \T$, we denote $A^f:=\{(x,s)\in \T^f\;:\; x\in A, \,0\leq s<f(x)\}$.

The following general remark follows from the definition of special flow and will be useful in the proof of Theorem \ref{th1}.

\begin{remark}\label{rem:specflow} Let $(G_t^h)$ be a special flow over an irrational rotation $G:\T\to \T$ and under $h\in {\rm BV}(\T)$, $h\in C^1(\T\setminus\{0\})$. For every $\epsilon>0$ there exists $\tilde{\delta}_\epsilon>0$ such that for every $(z,w),(z',w')\in \T^h$, $d^h((z,w),(z',w'))<\tilde{\delta}_\epsilon$, we have
$$
d^h\left(G^h_t(z,w), G^h_{t+h^{(N(z,w,t))}(z)-h^{(N(z,w,t))}(z')}(z',w')\right)<\epsilon^2,
$$
for every $t\in\R$ for which
$$
G^h_t(z,w)\notin \partial(\T^h,\epsilon),
$$
where
$\partial(\T^h,\epsilon):=\left\{\right(z,w)\in\T^h\,:\, \|z\|<\epsilon^2\text{ and } \epsilon^2<w<f(x)-\epsilon^2\}$.

Moreover, since $h\in BV(\T)\cap C^1(\T\setminus\{0\})$, for every $\epsilon>0$ there exists $\bar{\kappa}_\epsilon>0$ and $t_\epsilon>0$ such that for every $|T|\ge t_\epsilon$ and every $(z,w)\in\T^h$, we have 
\be\label{eq:largeprop}
\left|U_{z,w}\right|>(1-\epsilon^2)|T|,
\ee
where
$$
U_{z,w}:=\left\{t\in[0,T]: G^h_t(z,w)\notin \partial(\T^h,\bar\kappa_\epsilon)\right\}.
$$
Notice also that $U_{z,w}$ consists of at most $(\inf_\T h)^{-1}|T|$ disjoint intervals.
\end{remark}

\subsection{Diophantine lemmas}
Let $\beta\in \T$ and let $(q'_n)$ denote the sequence of denominators of $\beta$. We have the following lemma:
\begin{lemma}[\cite{KaSo},Lemma 3.3]\label{3d}
Fix $y,y'\in\T$, and let $n\in\N$ be any integer such that 
$$\|y-y'\|<\frac{1}{6 q_n'}.$$ 
Then at least one of the following holds:
\begin{itemize} 
\item[$(i)$] $0\notin\bigcup^{[\frac{q'_{n+1}}{6}]}_{k=0}R_\beta^k[y,y']$;
\item[$(ii)$] $0\notin\bigcup^{[\frac{q'_{n+1}}{6}]}_{k=0}R_\beta^{-k}[y,y']$;
\item[$(iii)$] $0\in\bigcup^{q'_{n}-1}_{k=0}R_\beta^{k}[y,y']$.
\end{itemize}
\end{lemma}

We will also need the following lemma, which is a simple consequence of Denjoy-Koksma inequality:

\begin{lemma}\label{lem:DK}Let $\phi:\T\to \R$ be a function of bounded variation. For every $\epsilon>0$ there exists $n_\epsilon\in \N$, such that for every $n\in \Z$ with $|n|\geq n_\epsilon$ and every $x\in \T$
$$
\left|\phi^{(n)}(x)-n\int_\T \phi d\lambda\right|<\epsilon |n|.
$$
\end{lemma}
\begin{proof}
By cocycle identity, it is enough to consider the case $n>0$. Notice that by Denjoy-Koksma inequality, there exists a constant $c_\phi>0$ such that for every $s\in \N$ and every $z\in \T$, we have
$$
\left|\phi^{(q_s)}(z)-q_s\int_\T \phi d\lambda\right|<c_\phi.
$$
The proof then follows by Ostrovski expansion since for $n\in \N$ we can write $n=\sum_{i=0}^kb_iq_i$, where $b_i\leq \frac{q_{i+1}}{q_i}$ and, by cocycle identity:
$$
\phi^{(n)}(x)=\sum_{i=0}^k\sum_{r=1}^{b_i} \phi^{(q_{k-i})}(x+z_{i,r}\alpha),
$$
where $z_{i,r}=\sum_{h=0}^ib_hq_h+rq_{k-i}$. This finishes the proof.
\end{proof}

\section{Disjointness criterion}\label{sec:disjoint}
We will use a variant of disjointness criterion introduced in \cite{KLU}.

\begin{proposition}\label{cri}
Let $P\subset \R^2\backslash\{(x,x):x\in\R\}$ be a compact subset and fix $c\in (0,1)$. Assume there exist $(A_k)\subset Aut(X_k,\cB|_{X_k},\lambda|_{X_k})$ for $k\geq 1$, such that $\lambda(X_k)\to \lambda(X)$,
$A_k\to Id$ uniformly.
Assume moreover that for every $\epsilon>0$ and $N\in \N$  
there exist $(E_k=E_k(\epsilon))\subset \cB$, $\lambda(E_k)\geq c\lambda(X)$, $0<\kappa=\kappa(\epsilon)<\epsilon$, $\delta=\delta(\epsilon,N)>0$, a set $Z=Z(\epsilon,N)\subset Y$, $\nu(Z)\geq (1-\epsilon)\nu(Y)$ such that for all $y,y'\in Z$ satisfying $d_2(y,y')<\delta$, every $k$ such that $d_1(A_k,Id)<\delta$ and every $x\in E_k$, $x':=A_{k}x$ there are $M\geq N$, $L\geq 1$, $\frac{L}{M}\geq \kappa$ and $(p, q)\in P$
 for which at least one of the following holds:
\begin{equation}\label{forw}
d_1(T_tx,T_{t+p}x'),\;d_2(S_ty,S_{t+q}y')<\epsilon\text{ for } t\in U\subset [M,M+L]
\end{equation}
or
\begin{equation}\label{backw}
d_1(T_{t}x,T_{t+p}x'),d_2(S_{t}y,S_{t+q}y')<\epsilon\text{ for } -t\in U\subset [M,M+L],
\end{equation} 
where $U$ is a union of at most $[c^{-1}L]+1]$ intervals and $|U|\geq (1-\epsilon)L$.
Then $(T_t)$ and $(S_t)$ are disjoint.
\end{proposition}
\begin{proof}The proof is a consequence of Theorem 3 in \cite{KLU}. Namely, for $x,x',y,y'$ as in the statement of Proposition \ref{cri} we define a function $a=a_{x,x',y,y'}:[M,M+L]\to \R$ given by $a(t)=t+q$. Then by Theorem 3 in \cite{KLU}, we just need to verify that $(a,U,c)$ is $\epsilon$-good (see Definition 3.1. in \cite{KLU}, $(P1)$). This however follows straightforwardly from the definition of $a$ and $U$. The proof is thus finished.
\end{proof}

\section{Proof of Theorem \ref{th1}}\label{proof:th1}
  In this section we will use Proposition \ref{cri} to prove Theorem \ref{th1}. For simplicity, we will use the following notation: $(T_t^f)$ is a special flow over $T(x)=x+\alpha$ (with sequence of denominators $(q_n)$) and under $f(x)=A_f\{x\}+f_{ac}(x)$, where $f_{ac}\in C^1(\T)$ and $(S_t^g)$ is a special flow over $S(x)=x+\beta$ (with sequence of denominators $(q'_n)$) and under $g(x)=A_g\{x\}+g_{ac}(x)$, where $g_{ac}\in C^1(\T)$. Moreover, we assume that $|A_f|\neq |A_g|$. Note that $(T^f_t)$ is isomorphic to $(R^{h}_t)_{t\in\R}$ built over $Rx=x-\alpha$ with roof function $h(x)=A_f\{1-x\}+f_{ac}(1-x)=(-A_f)\{x\}+(A_f+f_{ac}(1-x))$. This allows to assume that $A_f,A_g>0$. We will divide the proof in two cases depending on the diophantine types of $\alpha$ and $\beta$.

\subsection{Proof in case at least one of $\alpha$ and $\beta$ is of unbounded type}\label{sec:un}

We will without loss of generality assume that $\alpha$ is of unbounded type. 

\begin{proof}[Proof of Theorem \ref{th1}]
We will verify that the assumptions of Proposition \ref{cri} are satisfied. Since Proposition \ref{cri} has many quantifiers, we will divide the proof into paragraphs, in which we indicate what quantity we are defining.
 Let $\xi:=\frac{\int_\T g d\lambda}{\int_\T f d\lambda}$ and $\Delta:=1000\max(\xi,\frac{A_f}{A_g},\frac{A_g}{A_f},A_f^{-1},A_g^{-1})$.
 
 \textbf{Definition of $c$ and $P$:}
 Let $c:=\frac{\min\{A_f,A_g,|A_f-A_g|,\xi, \xi^{-1}\}}{100\Delta^2}\in (0,1)$ and
$$
P:=\left\{(p,q):\max(|p|,|q|)\le 10\max\{A_f,A_g\},\,|p-q|\ge c^2\right\}.$$

\textbf{Definition of $A_k$ and $X_k$:}
Recall that since $\alpha$ is of unbounded type, there exists an increasing sequence $\{n_k\}_{k\in\N}$ such $\frac{q_{n_k+1}}{q_{n_k}}\to \infty$.
Define 
$$X_k:=\left\{(x,s)\in \T^f \,:\, \left(x-\frac{c}{q_{n_k}},s\right)\in\T^f\right\},
$$ 
and $A_k(x,s)=(x-\frac{c}{q_{n_k}},s)$. It follows from the definition that $\lambda^f(X_k)\to \lambda(\T^f)$, $A_k\in Aut(X_k,\cB_{|X_k},\lambda^f_{|X_k})$ and $A_k\to Id$ uniformly.

\textbf{Definition of  $E_k(\epsilon)$:}
Fix $\epsilon<\min\{\frac{g_{\min}}{4},\frac{A_g}{72},\frac{g_{\max}A_g}{72},c\}$. 
Let 
$$
\tilde{E}_k:=\bigcup_{i=-n_k}^{-c^2q_{n_k}}T^i_{\alpha}\left[\frac{2}{q_{n_k+1}},\frac{c^2}{q_{n_k}}\right].
$$ 
Notice that $\lambda(\tilde{E}_k)\geq \frac{c^{4}}{2}$.  Let
$$
E_k=E_k(\epsilon)=\tilde{E}_k^f\cap X_k.
$$
Notice that $\lambda^f(\tilde{E}_k^f)\geq (\inf_\T f)\lambda(\tilde{E_k})$, and hence $\lambda^f(E_k) >c^5$. The following property of the set $E_k$ will be crucial in the proof: Let $(x,s)\in E_k$ and denote $(x',s)=A_k(x,s)$. Then there exists $i_x\in [n_k,c^2q_{n_k}]$ such that 
\be\label{eq:fbhit}
0\in [x+i_x\alpha, x'+i_x\alpha]\text{ and } 0\in [x-(q_{n_k}-i_x)\alpha, x'-(q_{n_k}-i_x)\alpha].
\ee
Indeed, notice that since $(x,s)\in E_k\subset \tilde{E}_k^f$ it follows that there exists $i_x\in [n_k,c^2q_{n_k}]$ such that $x+i_x\alpha\in [\frac{2}{q_{n_k+1}},\frac{c^2}{q_{n_k}}]$. By the definition of $A_k$, we have $x'+i_x\alpha=x+i_x\alpha-\frac{c}{q_{n_k}}$ and so $0\in [x'+i_x\alpha,x+i_x\alpha]$ (since, by assumptions $\frac{c}{q_{n_k}}>\frac{2}{q_{n_k+1}}$). Moreover, since 
$\|q_{n_k}\alpha\|\leq \frac{1}{q_{n_k+1}}$, we have 
$$x-(q_{n_k}-i_x)\alpha=x-i_x\alpha -q_{n_k}\alpha\in \left[\frac{1}{q_{n_k+1}}, \frac{c^2}{q_{n_k}}\right]$$ and 
$$x'-(q_{n_k}-i_x)\alpha= x+i_x\alpha-\frac{c}{q_{n_k}}-q_{n_k}\alpha<0.$$ This finishes the proof of \eqref{eq:fbhit}.

\textbf{Definition of $\kappa$, $\delta$ and $Z$:}
Let $\kappa:=\kappa(\epsilon)=c\epsilon^3$.

Notice that for every $k\in \Z$ there exists $\theta_k\in [x,y]$ such that $\psi^{(k)}(x)-\psi^{(k)}(y)=\psi'^{(k)}(\theta_k)(x-y)$, for $\psi\in \{f_{ac}, g_{ac}\}$.
By Lemma \ref{lem:DK} for $\phi=f_{ac}'$ and $\phi=g_{ac}'$, there exists $\delta_\epsilon>0$ such that if $z,z'\in \T$ satisfy $\|z-z'\|<\delta_\epsilon$, then
\be\label{eq:abscont}
|f^{(k)}_{ac}(z)-f^{(k)}_{ac}(z')|<\frac{\epsilon^3}{20}\max\{1,k\|z-z'\|\},\;|g^{(k)}_{ac}(z)-g^{(k)}_{ac}(z')|<\frac{\epsilon^3}{20}\max\{1,k\|z-z'\|\}.
\ee

Let $\delta=\delta(\epsilon,N):=\min\{\delta_\epsilon,\frac{\epsilon}{10},\frac{\epsilon}{20A_g}, \frac{12\epsilon}{A_gq'_{n_0}}\}$, where $n_0$ is such that 
$$ q'_{n_0}>\max\left\{12N(\inf_\T g)^{-1},N^2(\inf_\T g)^{-2}\right\}.$$
We will now define the set $Z=Z(\epsilon,N)$. First let
$$Z_1:=\{(y,r)\in\T^g:\delta<s<g(y)-\delta\}.$$
Notice that by the definition of $\delta$, we have $\lambda^g(Z_1)>(1-\epsilon/2)\lambda^g(\T^g)$. By the definition of $Z_1$ it follows that 
for every $(y,r),(y',r')\in Z_1$ with $d^g((y,r),(y',r'))<\delta$, we have $d^g((y,r),(y',r'))=\|y-y'\|+|r-r'|$.

Let $B_n:=\bigcup^{[\sqrt{q'_n}]}_{i=-[\sqrt{ q'_n}]}R^i[-\frac{1}{6 q'_n},\frac{1}{6q'_n}]$, then $\lambda(B_n)<\frac{1}{\sqrt{q'_n}}$. Since $\sum_{n=1}^\infty\frac{1}{\sqrt{q'_n}}<\infty$ ( $(q'_n)$ grows exponentially), there is an $m\in\N$ such that $\lambda(\bigcup_{n=m}^\infty B_n)<\frac{\epsilon}{4g_{\max}}$. Define 
$$Z_2:=\left\{(x,s)\in \T^g: x\notin\bigcup_{n=m}^\infty B_n\right\},$$
then $\lambda^g(Z_2)>(1-\epsilon/2)\lambda^g(\T^g)$. 
Finally, let $Z=Z_1\cap Z_2$. Notice that we have $\lambda^g(Z)>(1-\epsilon)\lambda^g(\T^g)$.

\textbf{Main estimates:}
Take $(x,s)\in E_k$ and $(x',s)=A_k(x,s)$, where $k\in \N$ is such that $d^f(A_k,Id)<\epsilon$ and let $(y,r),(y',r')\in Z$ with  $d^g((y,r),(y',r'))\le \delta$. 
Let $n\in\N$ be the unique integer such that 
\be\label{eq:disty}
\frac{1}{6 q'_{n+1}}\le \|y-y'\|<\frac{1}{6 q'_{n}}.
\ee

We will show that in the forward case, there exists an interval $[M',M'+L']$ with $\frac{L'}{M'}\geq \epsilon^3$ and such that for every $n\in [M',M'+L']$ and some $(p,q)\in P$, we have 
\be \label{shf}
|f^{(n)}(x)-f^{(n)}(x')-p|<\epsilon^2
\ee
and
\be\label{shg}
|g^{([\xi^{-1} n])}(y)-g^{([\xi^{-1} n])}(y')-q|<\epsilon^2.
\ee The backward case is analogous.

\textbf{Claim:} \eqref{shf} and \eqref{shg} (together with the backward version) imply the statement of Proposition \ref{cri}.

\begin{proof}[Proof of the {\bf Claim}] 
Let $$M:=\max(f^{(M')}(x),g^{([\xi^{-1} M'])}(y))$$ and $$L:=\min\left(f^{(M'+L')}(x)-f^{(M')}(x), g^{([\xi^{-1} (M'+L')])}(y)
- g^{([\xi^{-1} M'])}(y)\right).$$ Notice that by Remark \ref{rem:specflow} for $(T_t^f)$ and $(S_t^g)$ it follows that for every $t\in [M,M+L]$, for which 
$$
T_t^f(x,s)\notin \partial(\T^f,\epsilon^2)\text{ and } S_t^g(y,r)\notin \partial(\T^g,\epsilon^2), 
$$
we have by \eqref{shf} and \eqref{shg} that
$$
d^f(T^f_t(x,s),T_{t+p}^f(x',s'))<\epsilon\text{ and }
d^g(S^g_t(y,r),S_{t+q}^g(y',r'))<\epsilon.
$$
Moreover, by \eqref{eq:largeprop} for $(T_t^f)$ and $(S_t^g)$ it follows that if $U:=U_{x,s}\cap U_{y,r}$, then $|U|\geq (1-\epsilon)L$ and $U$ consists of at most $3\max((\inf_\T f)^{-1},(\inf_\T g)^{-1})$ intervals.

Hence it only remains to show that $\frac{L}{M}\geq \kappa$. Let $\xi_f=\int_\T fd\lambda$. By Lemma \ref{lem:DK} for $f$ and $g$, it follows that 
$$M\in [(1-\kappa^2)\xi_fM',(1+\kappa^2)\xi_f M']$$ and 
$$M+L\in [(1-\kappa^2)\xi_f(M'+L'),(1+\kappa^2)\xi_f (M'+L')].$$
Therefore (recall that $L'\geq \epsilon^3 M'$),
$$
\frac{M+L}{M}\geq \frac{1-\kappa^2}{1+\kappa^2}(1+\epsilon^3)\geq 1+\kappa,
$$
and so indeed $\frac{L}{M}\geq \kappa$. This gives \eqref{forw} in the statement of Proposition \ref{cri}.
\end{proof}

Hence we only need to show that \eqref{shf} and \eqref{shg} hold in the forward case. If $y=y'$, then the orbits of $(y,r)$ and $(y,r')$ will never diverge. In this case, we can let $M'=q_{n_k}$, $L'=\kappa M'$, and therefore \eqref{shf} and \eqref{shg} hold for $p=A_f$, $q=0$. So in the rest of the proof, assume $y\neq y'$. We will split the proof in three cases according to Lemma \ref{3d}.

\textbf{Case 1.} Assume $(i)$ in Lemma \ref{3d} holds for $y,y'$. 

By $(i)$ and \eqref{eq:abscont} for $g_{ac}$ and $y,y'$, we get that for any $i\in \left[1,\frac{ q'_{n+1}}{6}\right]$,
\be\label{eq:partg}
\left|g^{(i)}(y)-g^{(i)}(y')- iA_g(y-y')\right|\le |g_{ac}^{(i)}(y)-g_{ac}^{(i)}(y')|\le \frac{\epsilon^3}{20}\max\{1,i\|y-y'\|\}.
\ee

Now, let $\ell=\ell_{x,x'}\in\N$ be the smallest integer such that $0\in T_\alpha^\ell[x,x']$. Since $(x,s)\in E_k$ and by \eqref{eq:fbhit}, $\ell=i_x\in [n_k,c^2q_{n_k}]$. By the definition of $\ell$ it follows that for any $0\leq j<\ell$, $0\notin [T_\alpha^jx,T_\alpha^jx']$ and therefore and by \eqref{eq:abscont} for $f$ and $x,x'$, we get 
\be\label{eq:f1}
\left|f^{(j)}(x)-f^{(j)}(x')- jA_f(x-x')\right|\le\frac{\epsilon^3}{20}\max\{1,j\|x-x'\|\}<\frac{\epsilon^3}{10},
\ee
the last inequality, since (by the definition of $A_k$) $\|x-x'\|=\frac{c}{q_{n_k}}$ and $j\leq \ell\leq c^2q_{n_k}$.

Since $0\in [T^\ell_\alpha x,T^\ell_\alpha x']$ and $\|x-x'\|=\frac{c}{q_{n_k}}$
it follows that  for any $j\in [\ell+1,\ell+q_{n_k})$, $0\notin [T^j_\alpha x,T^j_\alpha x']$. Therefore for any $j\in [\ell+1,\ell+q_{n_k})$, using \eqref{eq:abscont} for $f$, and $x,x'$, we have
\be\label{eq:f2}
\left|f^{(j)}(x)-f^{(j)}(x')+ A_f-jA_f(x-x')\right|\le\frac{\epsilon^3}{20}\max\{1,j\|x-x'\|\}<\frac{\epsilon^3}{10},
\ee
where the last inequality since $j\leq \ell+q_{n_k}\le c^2q_{n_k}+q_{n_k}$ and $\|x-x'\|=\frac{c}{q_{n_k}}$. 
We consider the following subcases:

\textbf{Subcase 1.} $\ell A_g\|y-y'\|>4c$. Let then $\ell'\in \N$ be defined by $\ell':=\left[\frac{2c}{A_g\|y-y'\|}\right]$. By definition $\ell'\leq \frac{\ell}{2}$. Then by \eqref{eq:f1} for any $j\in[ \ell',(1+\epsilon^3) \ell']$, we have 
\be\label{fpart}
\left|f^{(j)}(x)-f^{(j)}(x')-\ell' A_f(x-x')\right|\leq \epsilon^3\ell'A_f\|x-x'\|+\epsilon^3/10\leq \epsilon^3\ell A_f\frac{c}{q_{n_k}}+\epsilon^3/10\leq \epsilon^2,
\ee
since $\ell \leq c^2q_{n_k}$. 
Notice that
\be\label{defp}
|p|:=\ell' A_f\|x-x'\|\leq \ell A_f\|x-x'\|\leq c^3A_f\leq c\xi^{-1}
\ee
 (by taking a smaller $c>0$ if necessary). 
Moreover, by \eqref{eq:partg} and the definition of $\ell'$, for $j\in [\ell', (1+\epsilon^3)\ell']$, we have 
\begin{multline}\label{eq:ooo}
\left|g^{([\xi^{-1}j])}(y)-g^{([\xi^{-1}j])}(y')-
[\xi^{-1}\ell']A_g(y-y')\right|\leq\\
\left([\xi^{-1}\ell']-[\xi^{-1}(1+\epsilon^3)\ell']\right)A_g\|y-y'\|+
\frac{\epsilon^3}{20}\max(1,[\xi^{-1}(1+\epsilon^3)\ell']\|y-y'\|)
\leq \epsilon^2,
\end{multline}
the last inequality by the definition of $\ell'$. We also have 
\be\label{defq}
|q|:=[\xi^{-1}\ell']A_g\|y-y'\|\in [2c\xi^{-1}-A_g\delta_{\epsilon},2c\xi^{-1}].
\ee
Notice that by \eqref{fpart}, \eqref{defp} and \eqref{eq:ooo}  \eqref{defq} it follows that \eqref{shf} and \eqref{shg} hold for $M'=\ell', L'=\epsilon^3M'$. Moreover, by \eqref{defp} and \eqref{defq} it follows that $(p,q)\in P$. This finishes the proof of \textbf{Subcase 1.}

\textbf{Subcase 2.} $\ell A_g\|y-y'\|\leq 4c$. 
Then by \eqref{eq:f2} for any $j\in[ \ell+1,(1+\epsilon^3) (\ell+1))$, we have (recall that $\ell\leq c^2q_{n_k}$)
\be\label{fs2}
\left|f^{(j)}(x)-f^{(j)}(x')+A_f- \ell A_f(x-x')\right|\leq \epsilon^3\ell
 A_f\|x-x'\|+\epsilon^3/10\leq \epsilon^3\ell A_f\frac{c}{q_{n_k}}+\epsilon^3/10\leq \epsilon^2
 \ee
Notice that if $p:=-A_f+\ell A_f(x-x')$, then 
\be\label{defp2}
-A_f\leq p\leq -A_f+\ell A_f\|x-x'\|\leq -A_f+c^2.
\ee
 Since $\ell A_g\|y-y'\|<4c$, we get (see \eqref{eq:disty}) $\xi^{-1}(1+\epsilon^3) (\ell+1)\leq \frac{q'_{n+1}}{6}$ and therefore, for every $j\in  [ \ell+1,(1+\epsilon^3) (\ell+1))$, we have
 \begin{multline}\label{eq:g2}
\left|g^{([\xi^{-1}j])}(y)-g^{(([\xi^{-1}j]))}(y')-
[\xi^{-1}\ell]A_g(y-y')\right|\leq\\
\left([\xi^{-1}\ell]-[\xi^{-1}(1+\epsilon^3)\ell]\right)A_g\|y-y'\|+
\frac{\epsilon^3}{20}\max(1,[\xi^{-1}(1+\epsilon^3)\ell]\|y-y'\|)
\leq \epsilon^2,
\end{multline}
 since $\ell A_g\|y-y'\|<4c$. Let $q:= [\xi^{-1}\ell]A_g(y-y')$, then
 \be\label{defq2}
 |q|\leq \xi^{-1}\ell A_g\|y-y'\|\leq 4c\xi^{-1}\leq \frac{A_f}{2},
 \ee
(by taking smaller $c>0$ if necessary). We define $M':=\ell+1$ and $L'=(1+\epsilon^3)M'$, then by \eqref{defp2} and \eqref{defq2}, we have that $(p,q)\in P$ and by \eqref{fs2} and \eqref{eq:g2} we get that \eqref{shf} and \eqref{shg} hold on $[M',M'+L']$. This finishes the proof of \textbf{Subcase 2. } and hence also the proof of \textbf{Case 1.}


\textbf{Case 2.} Assume (ii) holds. The proof is analogous to the proof in \textbf{Case 1}, by considering backward iterations.


\textbf{Case 3.} Assume (iii) holds. Let $k_0$ be the least real number such that $[\xi^{-1}k_0]\in [0,q'_n-1]$ and $0\in R^{[\xi^{-1}k_0]}[y,y']$. Then by the definition of $Z$, $\xi^{-1}k_0\ge \sqrt{ q'_n}$, and $0\notin R^k[y,y']$ for any $k\in [0,[\xi^{-1}k_0]+ q'_n], k\neq [\xi^{-1}k_0]$. Hence if $\xi^{-1}k<[\xi^{-1}k_0]$, by Lemma \ref{lem:DK}, for $\phi=g'_{ac}$ 
\be\label{eq:gfi}
\left||g^{([\xi^{-1}k])}(y)-g^{([\xi^{-1}k])}(y')|-[\xi^{-1}k]A_g\|y-y'\|\right|\le |g^{([\xi^{-1}k])}_{ac}(y)-g^{([\xi^{-1}k])}_{ac}(y')|\le \frac{\epsilon^3}{20}.
\ee
Moreover for 
$k_0+\xi<k<k_0+q_n'$, 
\be\label{eq:gfi2}
\left|g^{([\xi^{-1}k])}(y)-g^{([\xi^{-1}k])}(y')\pm A_g-[\xi^{-1}k]A_g(y-y')\right|\le |g^{([\xi^{-1}k])}_{ac}(y)-g^{([\xi^{-1}k])}_{ac}(y')|\le \frac{\epsilon^3}{20},
\ee
where $\pm$ depends only on $y,y'$.

Let $\ell\in\N$ be the smallest integer such that $0\in R^\ell[x,x']$. By the definition of $E_k$, $n_k\le \ell\le c^2q_{n_k}$. It is clear that for any $j<\ell$, 
\be\label{eq:fj}
\left|f^{(j)}(x)-f^{(j)}(x')- jA_f(x-x')\right|\le\frac{\epsilon^3}{20}\max\{1,j\|x-x'\|\}<\frac{\epsilon^3}{10},
\ee
 and for any $j\in [\ell,\ell+q_{n_k})$, 
 \be\label{eq:fj2}
 \left|f^{(j)}(x)-f^{(j)}(x')- A_f- jA_f(x-x')\right|\le\frac{\epsilon^3}{20}\max\{1,j\|x-x'\|\}<\frac{\epsilon^3}{10}.
  \ee

\textbf{Subcase 1.} $k_0\leq (1-\epsilon)\ell$.
notice that for every $T\in [0,(1+\epsilon^3)k_0]$ (note that $(1+\epsilon^3)k_0<\ell$)  and every $j\in [T,(1+\epsilon^3)T]$, by \eqref{eq:fj}, we have
\be\label{eq:fsm}
\left|f^{(j)}(x)-f^{(j)}(x')-TA_f(x-x')\right|\leq \epsilon^3\ell A_f\|x-x'\|+\epsilon^3/10\leq \epsilon^2,
\ee
and if $p(T):=TA_f(x-x')$, then 
\be\label{maxp}
|p(T)|\leq \ell A_f\|x-x'\|\leq c^3A_f.
\ee
Moreover,  by \eqref{eq:gfi} for every $j\in[(1-\epsilon^3)k_0,k_0]$, by \eqref{eq:disty} and $k_0\leq q'_n$, we have
\be\label{sub11}
\left||g^{([\xi^{-1}j])}(y)-g^{([\xi^{-1}j])}(y')|-[\xi^{-1}k_0]A_g\|y-y'\|\right|\le \xi^{-1}\epsilon^3k_0A_g\|y-y'\|+\epsilon^3/20\leq \epsilon^2.
\ee
Similarly, by \eqref{eq:gfi2}, for every $j\in[k_0+1,(1+\epsilon^3)(k_0+1)]$, by \eqref{eq:disty} and $k_0\leq q'_n$, we have
\be\label{sub12}
\left|g^{([\xi^{-1}k])}(y)-g^{([\xi^{-1}k])}(y')\pm (A_g-[\xi^{-1}k_0]A_g(y-y'))\right|\leq \epsilon^2.
\ee
Notice that if $q_1:=[\xi^{-1}k_0]A_g\|y-y'\|$ and $q_2:=\pm(A_g-[\xi^{-1}k_0]A_g(y-y'))$, then 
\be\label{maxq}
\max(|q_1|,|q_2|)\geq \frac{A_g}{2}.
\ee
By \eqref{eq:fsm}, \eqref{sub11} and \eqref{sub12} it follows that \eqref{shf} and \eqref{shg} hold on $$[M'_1,M'_1+L'_1]:=[[(1-\epsilon^3)k_0],(1-\epsilon^6)k_0]$$ with $p_1:=p((1-\epsilon^3)k_0)$ and $q_1$, and also on $$[M'_2,M'_2+L'_2]:=[k_0+1,(1+\epsilon^3)(k_0+1)]$$ with $p_2:=p(k_0+\xi)$ and $q_2$. Moreover, by \eqref{maxp} and \eqref{maxq}, at least one of $(p_1,q_1)$, $(p_2,q_2)$ belongs to $P$.

\textbf{Subcase 2.} $k_0\ge (1+\epsilon)\ell$. In this case let $M'_1:=\ell+1$ and $L'_1=\epsilon^3(\ell+1)$. Notice that $M'_1+L'_1<k_0$ and hence for every $j\in [M_1',M_1'+L_1']$ (reasoning analogously to \textbf{Case 1}), by \eqref{eq:fj2}, we have 
$$
\left|f^{(j)}(x)-f^{(j)}(x')-p_1\right|<\epsilon^2,
$$
where $p_1=-A_f+M_1'A_f(x-x')$ and similary by \eqref{eq:gfi}
$$
\left||g^{([\xi^{-1}j])}(y)-g^{([\xi^{-1}j])}(y')|-q_1\right|\le \epsilon^2,
$$
where $q_1:=[\xi^{-1}M_1']A_g(y-y')$.
Let $M'_2:=[\ell/2]$ and $L'_2=\epsilon^3M'_2$. Then 
 for every $j\in [M_2',M_2'+L_2']$ (reasoning analogously to \textbf{Case 1}), by \eqref{eq:fj}, we have 
$$
\left|f^{(j)}(x)-f^{(j)}(x')-p_2\right|<\epsilon^2,
$$
where $|p_2|=|-M_2'A_f(x-x')|\leq c^3A_f$ and similary by \eqref{eq:gfi}
$$
\left||g^{([\xi^{-1}j])}(y)-g^{([\xi^{-1}j])}(y')|-q_2\right|\le \epsilon^2,
$$
where $q_2:=[\xi^{-1}M_2']A_g(y-y')$. We will show that \eqref{shf} and \eqref{shg} hold either on $[M'_1,M'_1+L'_1]$ or on  $[M'_2,M'_2+L'_2]$.

 By the above, we only need to show that at least one  of $(p_1,q_1)$ and $(p_2,q_2)$ belongs to $P$.
Notice that $q_2\geq \frac{q_1}{3}$, $p_1\geq 10p_2$ and $|p_1|\ge \frac{A_f}{2}$ (since $c>0$ is small). Therefore if $|p_2-q_2|<c^2$, then $|p_1-q_1|>c^2$. This finishes the proof in \textbf{Subcase 2.}

\textbf{Subcase 3.} $(1-\epsilon)\ell\le k_0\le (1+\epsilon)\ell$.
The proof is similar to the proof in \textbf{Subcase 2.}.

Let $M_1:=(1-\epsilon)\min(k_0,\ell)$ and $L_1:=\epsilon^3M_1$,
$M_2:=(1+\epsilon)\max(k_0,\ell)$ and $L_2:=\epsilon^3M_2$. 
Notice that by \eqref{eq:fj} for every $j\in [M_1',M_1'+L_1']$, we have 
\be\label{l1}
\left|f^{(j)}(x)-f^{(j)}(x')-p_1\right|<\epsilon^2,
\ee
where $p_1=M_1A_f(x-x')$ and similary by \eqref{eq:gfi}
\be\label{l2}
\left||g^{([\xi^{-1}j])}(y)-g^{([\xi^{-1}j])}(y')|-q_1\right|\le \epsilon^2,
\ee
where $q_1:=[\xi^{-1}M_1]A_g(y-y')$.
Similarly, by \eqref{eq:fj2} for every $j\in [M_2,M_2+L_2]$, we have 
\be\label{l3}
\left|f^{(j)}(x)-f^{(j)}(x')-p_1\right|<\epsilon^2,
\ee
where $p_2=-A_f+M_2A_f(x-x')$ and similary by \eqref{eq:gfi}
\be\label{l4}
\left||g^{([\xi^{-1}j])}(y)-g^{([\xi^{-1}j])}(y')|-q_1\right|\le \epsilon^2,
\ee
where $q_2=\pm(A_g-[\xi^{-1}M_2]A_g(y-y'))$.

Notice that $|p_1|<c^3A_f$, $|q_1|<\xi^{-1}A_g$ and $|p_2|<A_f+c^3A_f$ and $|q_2|<A_g+\xi^{-1}A_g$. If
 $(p_1,q_1)\in P$, then \eqref{shf} and \eqref{shg} hold on $[M_1,M_1+L_1]$ (see \eqref{l1} and \eqref{l2}). If $(p_1,q_1)\notin P$, then
$$
|p_1-q_1|<c^2,
$$
and so $|q_1|<2c^2\max(1,A_f)$. We also have $|q_2-A_g|\leq 2|q_1|\leq 4c^2\max(1,A_f)$. Therefore 
$$
|p_2-q_2|\geq |A_f-A_g|-c^2A_f-4c^2\max(1,A_f)\ge \frac{1}{2} |A_f-A_g|\ge c.
$$
Therefore $(p_2,q_2)\in P$ and \eqref{shf} and \eqref{shg} hold on $[M_2,M_2+L_2]$ (see \eqref{l3} and \eqref{l4}).

This finishes the proof of \textbf{Subcase 3.} and hence also the proof of Theorem \ref{th1}.
\end{proof}

\subsection{The proof when both $\alpha$ and $\beta$ are of bounded type}\label{sec:bo}

The argument here is similar to that in previous subsection. Recall that $A_f>0, A_g>0$. 

\begin{proof}[Proof of Theorem \ref{th1}]
We will verify that the assumptions of Proposition \ref{cri} are satisfied. Let $\xi:=\frac{\int_\T g d\lambda}{\int_\T f d\lambda}$ and $\Delta:=1000\max(\xi,\frac{A_f}{A_g},\frac{A_g}{A_f},A_f^{-1},A_g^{-1})$.
 
 \textbf{Definition of $c$ and $P$:}
 Let $a_0:=10\max\left(\sup\{\frac{q_{n+1}}{q_n}\},\sup\{\frac{q'_{n+1}}{q'_n}\}\right)$, and $$c:=\frac{\min\{A_f,A_g,a_0^{-1},|A_f-A_g|,\xi,\xi^{-1}\}}{100\Delta^2}>0.$$
 Let
$
P:=\left\{(p,q):\max(|p|,|q|)\le 10\max\{A_f,A_g\},\,|p-q|\ge c^2\right\}.$


\textbf{Definition of $A_k$ and $X_k$:}
Since $\alpha$ is of bounded type, we have $\frac{q_{n+1}}{q_{n}}<\frac{1}{2c}$ for all $n$.
Define 
$$
X_k:=\left\{(x,s)\in \T^f \;:\; (x-\frac{c}{q_{k}},s)\in\T^f\right\},
$$ 
and $A_k(x,s)=(x-\frac{c}{q_{k}},s)$. It follows from the definition that $\lambda^f(X_k)\to \lambda^f(\T^f)$, $A_k\in Aut(X_k,\cB_{|X_k},\lambda^f_{|X_k})$ and $A_k\to Id$ uniformly.

\textbf{Definition of  $E_k(\epsilon)$:}
Fix $\epsilon<\min\{\frac{g_{\min}}{4},\frac{A_g}{72},\frac{g_{\max}A_g}{72},c\}$. 
Let 
$$
\tilde{E}_k:=\bigcup_{i=-k}^{-c^2q_{k}}T^i_{\alpha}\left[\frac{c^2}{q_{k}},\frac{2c^2}{q_{k}}\right].
$$ 
Notice that $\lambda(\tilde{E}_k)\geq \frac{c^{4}}{2}$.  Let
$$
E_k:=E_k(\epsilon)=\tilde{E}_k^f\cap X_k.
$$
Notice that $\lambda^f(\tilde{E}_k^f)\geq (\inf_\T f)\lambda(\tilde{E_k})$, and hence $\lambda^f(E_k) >c^5$. By the definition of $E_k\subset \tilde{E}_k^f$ it follows that if $(x,s)\in E_k$ and $(x',s)=A_k(x,s)$, then there exists a unique $i_x\in [k,c^2q_{k}]$ such that 
\be\label{eq:fbhit-1}
0\in [x+i_x\alpha, x'+i_x\alpha].
\ee
\textbf{Definition of $\kappa$, $\delta$ and $Z$:} Let $\kappa:=\kappa(\epsilon)=c\epsilon^3$.

Notice that for every $k\in \Z$ there exists $\theta_k\in [x,y]$ such that $\psi^{(k)}(x)-\psi^{(k)}(y)=\psi'^{(k)}(\theta_k)(x-y)$, for $\psi\in \{f_{ac}, g_{ac}\}$.
By Lemma \ref{lem:DK} for $\phi=f_{ac}'$ and $\phi=g_{ac}'$, there exists $\delta_\epsilon>0$ such that if $z,z'\in \T$ satisfy $\|z-z'\|<\delta_\epsilon$, then
\be\label{eq:abscont-1}
|f^{(k)}_{ac}(z)-f^{(k)}_{ac}(z')|<\frac{\epsilon}{20}\max\{1,k\|z-z'\|\},\;|g^{(k)}_{ac}(z)-g^{(k)}_{ac}(z')|<\frac{\epsilon}{20}\max\{1,k\|z-z'\|\}.
\ee

Let $\delta=\delta(\epsilon,N):=\min\{\delta_\epsilon,\frac{\epsilon}{10},\frac{\epsilon}{20A_g}, \frac{12\epsilon}{A_g q'_{n_0}}\}$, where $n_0$ is such that 
$$q'_{n_0}>\max\left\{12N(\inf_\T gd\lambda)^{-1},N^2(\inf_\T gd\lambda)^{-2}\right\}.$$
We will now define the set $Z=Z(\epsilon,N)$. First let
$$Z_1:=\{(y,r)\in\T^g:\delta<s<g(y)-\delta\}.$$
Notice that by the definition of $\delta$, we have $\lambda^g(Z_1)>(1-\epsilon/2)\lambda^g(\T^g)$. By the definition of $Z_1$ it follows that 
for every $(y,r),(y',r')\in Z_1$ with $d^g((y,r),(y',r'))<\delta$, we have $d^g((y,r),(y',r'))=\|y-y'\|+|r-r'|$.

Let $B_n:=\bigcup^{[\sqrt{q'_n}]}_{i=-[\sqrt{q'_n}]}R^i[-\frac{1}{6  q'_n},\frac{1}{6 q'_n}]$, then $\lambda(B_n)<\frac{1}{\sqrt{ q'_n}}$. Since $\sum_{n=1}^\infty\frac{1}{\sqrt{q'_n}}<\infty$ ( $( q'_n)$ grows exponentially), there is an $m\ge N$ such that $\lambda(\bigcup_{n=m}^\infty B_n)<\frac{\epsilon}{4g_{\max}}$. Define 
$$Z_2:=\left\{(x,s)\in \T^g: x\notin\bigcup_{n=m}^\infty B_n\right\},$$
then $\lambda^g(Z_2)>(1-\epsilon/2)\lambda^g(\T^g)$. 
Finally, let $Z=Z_1\cap Z_2$. Notice that we have $\lambda^g(Z)>(1-\epsilon)\lambda^g(\T^g)$.

\textbf{Main estimates:}
Take $(x,s)\in E_k$ and $(x',s)=A_k(x,s)$, where $k\in \N$ is such that $d^f(A_k,Id)<\epsilon$ and let $(y,r),(y',r')\in Z$ with  $d^g((y,r),(y',r'))\le \delta$. 
Let $n\in\N$ be the unique integer such that 
\be\label{eq:disty-1}
\frac{1}{2 q'_{n+1}}\le \|y-y'\|<\frac{1}{2 q'_{n}}.\ee


Let $\ell_0$ be the smallest positive integer that $\{\ell_0\alpha\}\in [y,y']$. It is clear by the definition of $Z$ that $\sqrt{q'_n}\le \ell_0\le  q'_{n+1}$, and $\{j\alpha\}\notin [y,y']$ for any $j\in [1,\ell_0+ q'_n)$ with $j\neq \ell_0$. Hence if $j\in [1,\ell_0-1]$, by \eqref{eq:abscont-1}
\be \label{arg1}
\left||g^{(j)}(y)-g^{(j)}(y')|-jA_g\|y-y'\|\right|\le |g^{(j)}_{ac}(y)-g^{(j)}_{ac}(y')|\le \frac{\epsilon^3}{20};
\ee
and if $\ell_0<j<\ell_0+\bar q_n$, 
\be \label{arg2}
\left||g^{(j)}(y)-g^{(j)}(y')\pm A_g|-jA_g\|y-y'\|\right|\le |g^{(j)}_{ac}(y)-g^{(j)}_{ac}(y')|\le \frac{\epsilon^3}{20},
\ee 
where $\pm$ depends only on $y,y'$.

For $(x,s),(x',s)$, let $m\in\N$ be the smallest integer such that $0\in R^m[x,x']$. By the construction of $E_i$, $n_k\le m\le c^2q_{n_k}$. Similarly, by \eqref{eq:abscont-1} for any $j\in[1,m-1]$, 
\be\label{e-1-1}
\left|f^{(j)}(x)-f^{(j)}(x')- jA_f(x-x')\right|\le\frac{\epsilon}{20}\max\{1,j\|x-x'\|\}<\frac{\epsilon^3}{20},
\ee
and for any $j\in [m+1,m+q_{n_k})$, 
\be\label{e-1-2}
\left|f^{(j)}(x)-f^{(j)}(x')- A_f- jA_f(x-x')\right|\le\frac{\epsilon}{20}\max\{1,j\|x-x'\|\}<\frac{\epsilon^3}{10}.
\ee

Reasoning as in the previous subsection (see \eqref{shf}, \eqref{shg} and the proof of the \textbf{Claim}) it suffices to show that there exists an interval $[M',M'+L']$ with $\frac{L'}{M'}\geq \epsilon^3$, such that for every $n\in [M',M'+L']$ and some $(p,q)\in P$, we have 
\be \label{shf-1}
\left|f^{(n)}(x)-f^{(n)}(x')-p\right|<\epsilon^2
\ee
and
\be\label{shg-1}
\left|g^{([\xi^{-1} n])}(y)-g^{([\xi^{-1} n])}(y')-q\right|<\epsilon^2.
\ee

Due to the same reason as the previous unbounded type case, assume further $y\neq y'$. We consider  three cases, which are analogous to \textbf{Subcase 1}--\textbf{Subcase 3}  in {\bf Case 3.} in Subsection \ref{sec:un}.


\textbf{A.} $\ell_0\le (1-\epsilon)m\xi^{-1}$. Notice that for every $j\in[(1-\epsilon^3)\xi\ell_0,\xi\ell_0)$ and $ j\in(\xi\ell_0,(1+\epsilon^3)\xi\ell_0]$, (note that $(1+\epsilon^3)\xi\ell_0<m$) by \eqref{e-1-1}, we have
\be\label{eq:fsm-1}
\left|f^{(j)}(x)-f^{(j)}(x')-\xi\ell_0A_f(x-x')\right|\leq \epsilon^3\xi\ell_0 A_f\|x-x'\|+\epsilon^3/10\leq \epsilon^2,
\ee
and if $p:=\xi\ell_0A_f(x-x')$, then 
\be\label{maxp-1}
|p|\leq \xi\ell_0 A_f\|x-x'\|\leq c^3A_f.
\ee
Moreover,  by \eqref{arg1} and \eqref{eq:disty-1}, for every $j\in[(1-\epsilon^3)\xi\ell_0,\xi\ell_0)$, and the fact $\ell_0\leq a_0q'_n$, we have
\be\label{sub11-1}
\left||g^{([\xi^{-1}j])}(y)-g^{([\xi^{-1}j])}(y')|-[\xi\ell_0]A_g\|y-y'\|\right|\le \xi\epsilon^3\ell_0A_g\|y-y'\|+\epsilon^3/20\leq \epsilon^2.
\ee
Similarly, by \eqref{arg2}and \eqref{eq:disty-1}, for every $j\in[\xi\ell_0+1,(1+\epsilon^3)\xi\ell_0]$, and $\ell_0\leq a_0q'_n$, we have
\be\label{sub12-1}
\left|g^{([\xi^{-1}j])}(y)-g^{([\xi^{-1}j])}(y')\pm A_g-[\xi\ell_0]A_g(y-y')\right|\leq \epsilon^2.
\ee
Notice that if $q_1:=[\xi\ell_0]A_g\|y-y'\|$ and $q_2:=\pm A_g-[\xi\ell_0]A_g(y-y')$, then 
\be\label{maxq-1}
\max(|q_1|,|q_2|)\geq A_g/2.
\ee
 By \eqref{eq:fsm-1}, \eqref{sub11-1} and \eqref{sub12-1} it follows that \eqref{shf-1} and \eqref{shg-1} hold with $p,q_1$ on $$[M'_1,M'_1+L'_1]=[[(1-\epsilon^3)\xi\ell_0],\xi\ell_0-1]$$ and also with $p,q_2$ on $$ [M'_2,M'_2+L'_2]=[[\xi\ell_0]+1,(1+\epsilon^3)\xi\ell_0].$$ By \eqref{maxp-1} and \eqref{maxq-1}, at least one of $(p,q_1)$, $(p,q_2)$ belongs to $P$.


\textbf{B.} $\ell_0\ge (1+\epsilon)m\xi^{-1}$.

In this case let $M'_1:=m+1$ and $L'_1=\epsilon^3m-1$. Notice that $M'_1+L'_1<\xi\ell_0$ and hence for every $j\in [M_1',M_1'+L_1']$ (reasoning analogously to \textbf{A.}), by \eqref{e-1-2}, we have 
$$
\left|f^{(j)}(x)-f^{(j)}(x')-p_1\right|<\epsilon^2,
$$
where $p_1:=A_f-M_1'A_f(x-x')$ and similarly by \eqref{arg1}
$$
\left||g^{([\xi^{-1}j])}(y)-g^{([\xi^{-1}j])}(y')|-q_1\right|\le \epsilon^2,
$$
where $q_1:=[\xi^{-1}M_1']A_g(y-y')$.
Let $M'_2:=[m/2]$ and $L'_2=\epsilon^3M'_2$. Then 
 for every $j\in [M_2',M_2'+L_2']$ (reasoning analogously to \textbf{A.}), by \eqref{e-1-1}, we have 
$$
\left|f^{(j)}(x)-f^{(j)}(x')-p_2\right|<\epsilon^2,
$$
where $p_2=-M_2'A_f(x-x')\leq c^3A_f$ and similary by \eqref{arg1}
$$
\left||g^{([\xi^{-1}j])}(y)-g^{([\xi^{-1}j])}(y')|-q_2\right|\le \epsilon^2,
$$
where $q_2:=[\xi^{-1}M_2']A_g(y-y')$. We will show that \eqref{shf-1} and \eqref{shg-1} hold either on $[M'_1,M'_1+L'_1]$ or on  $[M'_2,M'_2+L'_2]$.

 By the above, we only need to show that at least one  of $(p_1,q_1)$ and $(p_2,q_2)$ belongs to $P$.
Notice that $|q_2|\geq \frac{|q_1|}{3}$, $|p_1|\geq 10|p_2|$ and $|p_1|\geq \frac{A_f}{2}$ (since $c>0$ is small). Therefore if $|p_2-q_2|<c^2$, then $|p_1-q_1|\ge c^2$. This finishes the proof in \textbf{B.}.

\textbf{C.} $(1-\epsilon)m\xi^{-1}\le \ell_0\le (1+\epsilon)m\xi^{-1}$.
The proof is similar to the proof in \textbf{B.}.

Let $M_1':=(1-\epsilon)\min(\xi\ell_0,m)$ and $L_1':=\epsilon^3M_1'$,
$M_2':=(1+\epsilon)\max(\xi\ell_0,m)$ and $L_2':=\epsilon^3M_2'$. 
Notice that for every $j\in [M_1',M_1'+L_1']$, by \eqref{e-1-1} we have 
\be\label{l1-1}
\left|f^{(j)}(x)-f^{(j)}(x')-p_1\right|<\epsilon^2,
\ee
where $p_1=M_1'A_f(x-x')$ and similary by \eqref{arg1}
\be\label{l2-1}
\left||g^{([\xi^{-1}j])}(y)-g^{([\xi^{-1}j])}(y')|-q_1\right|\le \epsilon^2,
\ee
where $q_1:=[\xi^{-1}M_1']A_g(y-y')$.
Similarly, for every $j\in [M_2',M_2'+L'_2]$, by \eqref{e-1-2} we have 
\be\label{l3-1}
\left|f^{(j)}(x)-f^{(j)}(x')-p_1\right|<\epsilon^2,
\ee
where $p_2=-A_f+M_2'A_f(x-x')$ and similary by \eqref{arg2}
\be\label{l4-1}
\left||g^{([\xi^{-1}j])}(y)-g^{([\xi^{-1}j])}(y')|-q_1\right|\le \epsilon^2,
\ee
where $q_2=\pm A_g-[\xi^{-1}M'_2]A_g(y-y')$.

Notice that $|p_1|<c^3A_f$, $|q_1|<\xi^{-1}A_g$ and $|p_2|<A_f+c^3A_f$ and $|q_2|<A_g+\xi^{-1}A_g$. If
 $(p_1,q_1)\in P$, then \eqref{shf-1} and \eqref{shg-1} hold on $[M'_1,M'_1+L'_1]$ (see \eqref{l1-1} and \eqref{l2-1}). If $(p_1,q_1)\notin P$, then
$$
|p_1-q_1|<c^3,
$$
and so $|q_1|<2c^3\max(1,A_f)$. We also have $||q_2|-A_g|\leq 2|q_1|\leq 4c^3\max(1,A_f)$. Therefore 
$$
\big||p_2|-| q_2|\big|\geq |A_f-A_g|-c^3A_f-4c^3\max(1,A_f)>\frac{1}{2}|A_f-A_g|\geq c^2,
$$
if $c$ is small enough. Therefore $(p_2,q_2)\in P$ and \eqref{shf-1} and \eqref{shg-1} hold on $[M_2,M_2+L_2]$ (see \eqref{l3-1} and \eqref{l4-1}).

This finishes the proof of \textbf{C.} and hence also the proof of Theorem \ref{th1}.
\end{proof}

\section{Proof of Corollary \ref{cor1}}\label{proof:cor}

\begin{proof}[Proof of Corollary \ref{cor1}]
Let $\cD:=DC$ (see \eqref{dioph}) and let $\alpha\in\cD$ and $\phi,\psi\in C^{\infty}(\T^2)$. Recall that for $\tau\in\{\psi,\phi\}$, the flow $(v_t^{\alpha,\tau})$ is isomorphic to $(T^{\alpha,f_\tau}_t)$, where $A_{f_\tau}=\int_{\partial D}\tau(v_t^\alpha)dt>0$. Moreover, $\psi,\phi$ are cohomologous for $(v_t^\alpha$) if and only if $f_\psi, f_{\phi}$ are cohomologous for $R_\alpha$\footnote{This follows from the fact that $(f_{\tau},R_\alpha)$ is the first return map and the first return time of the flow $v_t^{\alpha,\tau}$, $\tau\in\{\psi,\phi\}$}. By Theorem \ref{th1} and \cite{Kol}, it follows that the latter holds if and only if $A_{f_\psi}=A_{f_\phi}$. Indeed, the fact that it is neccesary follows from  Theorem \ref{th1}, and that it is sufficient from \cite{Kol}, as in this case (since the jumps cancel out)
$$
f_{\psi}-f_{\phi}\in C_0^{\infty}(\T),
$$
here we also use the fact that $\int_{\T^2}\phi d\lambda_{\T^2}=\int_{\T^2}\psi d\lambda_{\T^2}$ to know that $\int_{\T}(f_{\psi}-f_{\phi})d\lambda_{\T}=0$.
\end{proof}

\textbf{Acknowledgements.} The authors would like to thank Krzysztof Fr\c aczek for several discussions on the subject.

\end{document}